\documentclass[11pt,oneside,english,reqno]{amsart}
\usepackage[T1]{fontenc}
\usepackage[latin9]{inputenc}
\usepackage{geometry}
\usepackage[active]{srcltx}
\usepackage{babel}
\usepackage{amsthm}
\usepackage{amssymb}
\usepackage{esint}
\usepackage[unicode=true,pdfusetitle,
 bookmarks=true,bookmarksnumbered=false,bookmarksopen=false,
 breaklinks=false,pdfborder={0 0 1},backref=false,colorlinks=false]
 {hyperref}
\usepackage{breakurl}

\makeatletter
\numberwithin{equation}{section}
\numberwithin{figure}{section}
\theoremstyle{plain}
\newtheorem{thm}{\protect\theoremname}

\makeatother

\providecommand{\theoremname}{Theorem}

\begin{document}

\title[Landau -- Hadamard type inequality for mappings]{A counterpart of Landau \textendash{} Hadamard type inequality\\ for
manifold-valued mappings}

\author{Igor Parasyuk}

\address{Faculty of Mechanics and Mathematics, Taras Shevchenko National University
of Kyiv, 64/13, Volodymyrska Street, City of Kyiv, Ukraine, 01601 }

\email{pio@univ.kiev.ua}

\subjclass[2010]{26D10; 26D20; 53C21}

\begin{abstract}
We obtain a Landau -- Hadamard type inequality
for mappings defined on the whole real axis and taking values in Riemannian manifolds. In terms of an auxiliary convex function, we find conditions under which the boundedness of  covariant derivative along the curve under consideration ensures the boundedness of the corresponding tangent vector field.  As an example we obtain a Landau -- Hadamard type inequality for curves on 2D unite sphere.
\end{abstract}

\keywords{ Riemannian manifold; covariant derivative; Landau -- Hadamard inequality}

\maketitle

\section{Introducton}

In the present paper, we intend to obtain a counterpart of Landau
\textendash{} Hadamard inequality for curves on Riemannian manifolds.
In what follows we use the following notations:
\begin{itemize}
\item $\left(\mathcal{M},\mathfrak{g}=\left\langle \cdot,\cdot\right\rangle \right)$
is a smooth complete Riemannian manifold with metric tensor $\mathfrak{g}$;
\item $\nabla$ is the Levi-Civita connection with respect to $\mathfrak{g}$;
\item $\rho(\cdot,\cdot):\mathcal{M}\times\mathcal{M}\mapsto\mathbb{R}_{+}$
is the corresponding distance function;
\item $\left\Vert \cdot\right\Vert $ is the norm associated with the inner
product $\left\langle \cdot,\cdot\right\rangle $ on tangent spaces
$T_{x}\mathcal{M}$, $x\in\mathcal{M}$;
\item $\dot{x}(t)\in T_{x(t)}\mathcal{M}$ is the tangent vector for the
mapping ${x}(\cdot)$ at point $x(t)$
\item $\nabla_{\dot{x}}\xi(t)$ is the covariant derivative of a smooth
vector field $\xi(\cdot):I\mapsto T\mathcal{M}:=\bigsqcup_{x\in\mathcal{M}}T_{x}\mathcal{M}$
along a smooth mapping $x(\cdot):I\mapsto\mathcal{M}$;
\item $\nabla_{\xi}X(x)$ is the covariant derivative of a vector field
$X$ at point $x\in\mathcal{M}$ in direction of a tangent vector
$\xi\in T_{x}\mathcal{M}$,
\item $\nabla U(x)$ is the gradient of a smooth function $U(\cdot):\mathcal{M}\mapsto\mathbb{R}$
at point $x\in\mathcal{M}$.
\end{itemize}
When studying on a Riemnnian manifold the existence problem for bounded
solutions to Newtonian equation
\begin{gather*}
\nabla_{\dot{x}}\dot{x}=F(x)
\end{gather*}
with smooth force field $F$, there naturally arises the question:
is it true that the boundedness of a solution $x(\mathbb{\cdot}):\mathbb{R}\mapsto\mathcal{M}$,
and thus the finiteness of $\sup_{t\in\mathbb{R}}\left\Vert \nabla_{\dot{x}}\dot{x}(t)\right\Vert $,
always ensures the boundedness of $\left\Vert \dot{x}(\cdot)\right\Vert $
on $\mathbb{R}$ (see, e.g., \cite{Par17})? As was observed in \cite{Cie03},
in the case where $\mathcal{M}=\mathbb{R}^{d}$, the answer is positive
in view of the Landau\textendash{}Hadamard inequality (see, e.g.,\cite{MPF91,Finch03}).
Recall that if a function $f(\cdot)\in\mathrm{C}^{2}\left(\mathbb{R}\!\mapsto\!\mathbb{R}\right)$
satisfies the conditions
\begin{gather*}
\left\Vert f(\cdot)\right\Vert _{\infty}:=\sup_{t\in\mathbb{R}}\left|f(t)\right|<\infty,\quad\left\Vert f^{(2)}(\cdot)\right\Vert _{\infty}<\infty,
\end{gather*}
where $\left\Vert \cdot\right\Vert _{\infty}:=\sup_{t\in\mathbb{R}}\left\Vert \cdot\right\Vert $, then the Landau \textendash{} Hadamard inequality reads
\begin{gather*}
\left\Vert f^{\prime}(\cdot)\right\Vert _{\infty}^{2}\le2\left\Vert f(\cdot)\right\Vert _{\infty}\left\Vert f^{\prime\prime}(\cdot)\right\Vert _{\infty}.
\end{gather*}

G.~A.~Anastassiou \cite{Anas12,Anas16} studied the case of Banach
space valued functions and showed that the analogous inequality holds
true, although with constant $4$ instead of $2$.

It is not hard to construct an example of curve on a sphere with bounded
$\left\Vert \nabla_{\dot{x}}\dot{x}(\cdot)\right\Vert _{\infty}$
and unbounded $\left\Vert \dot{x}(\cdot)\right\Vert _{\infty}$.

In order to obtain a counterpart of Landau \textendash{} Hadamard
inequality we use an auxiliary function with positive definite Hessian
\cite{Par17_1}. In the case of Riemannian manifold, such a function
can be considered as a surrogate for the square of norm in $\mathbb{R}^{d}$.

\section{A Landau \textendash{} Hadamard type inequality on Riemannian manifold\label{sec:Convexfunct}}

We obtain the following estimate for $\dot{x}(\cdot)$ in terms of an
auxiliary function $U(\cdot)$ and the covariant derivative $\nabla_{\dot{x}}x(\cdot)$.
\begin{thm}
\label{thm:1}Let $x(\cdot):\mathbb{R}\mapsto\mathcal{M}$ be a smooth
mapping such that $\left\Vert \nabla_{\dot{x}}\dot{x}(\cdot)\right\Vert _{\infty}<\infty.$
Suppose that there exists a smooth function $U(\cdot):\mathcal{M}\mapsto\mathbb{R}$
satisfying the inequalities
\begin{gather*}
\sup_{t\in\mathbb{R}}U\circ x(t)<\infty,\quad0<\left\Vert \nabla U\circ x(\cdot)\right\Vert _{\infty}<\infty,\\
\lambda:=\inf_{t\in\mathbb{R}}\min\left\{ \left\langle \nabla_{\xi}\nabla U(x(t)),\xi\right\rangle :\xi\in T_{x(t)}\mathcal{M},\;\left\Vert \xi\right\Vert =1\right\} >0.
\end{gather*}
Then
\begin{gather*}
\left\Vert \dot{x}(\cdot)\right\Vert _{\infty}^{2}\le\frac{C^{2}}{\lambda}\left\Vert \nabla U\circ x(\cdot)\right\Vert _{\infty}\left\Vert \nabla_{\dot{x}}\dot{x}(\cdot)\right\Vert _{\infty}
\end{gather*}
 where the constant $C$ does not exceed the positive root of the
equation $\zeta^{3}-3\zeta-1=0.$ \end{thm}
\begin{proof}
In what follows we will use the notations
\begin{gather*}
u(t):=U\circ x(t),\quad v(t):=\dot{u}(t)\equiv\left\langle \nabla U(x(t)),\dot{x}(t)\right\rangle ,\\
r_{0}:=\left\Vert \nabla U\circ x(\cdot)\right\Vert _{\infty},r_{2}:=\left\Vert \nabla_{\dot{x}}\dot{x}(\cdot)\right\Vert _{\infty}.
\end{gather*}
Since $\left|v(t)\right|\le r_{0}\left\Vert \dot{x}(t)\right\Vert $,
then
\begin{gather*}
\begin{split}\dot{v}(t) & =\left\langle \nabla_{\dot{x}}H(x(t)),\dot{x}(t)\right\rangle +\left\langle \nabla U(x(t)),\nabla_{\dot{x}}\dot{x}(t)\right\rangle \\
 & \ge\lambda\left\Vert \dot{x}(t)\right\Vert ^{2}-r_{0}r_{2}\ge\frac{\lambda}{r_{0}^{2}}v^{2}(t)-r_{0}r_{2}.
\end{split}
\end{gather*}
Let us show that $v^{2}(t)\le r_{0}^{3}r_{2}/\lambda$ for all $t\in\mathbb{R}$.
In fact, if there exists $t_{0}$ such that $v(t_{0})>\sqrt{r_{0}^{3}r_{2}/\lambda}$,
then $v(t)$ increases for $t\ge t_{0}$, and $\dot{v}(t)\ge\lambda v^{2}(t_{0})/r_{0}^{2}-r_{0}r_{2}>0$.
Thus $v(t)\to+\infty$ and we arrive at contradiction: $u(t)\to+\infty$
as $t\to+\infty$. Now suppose that there exists $t_{0}$ such that
$v(t_{0})<-\sqrt{r_{0}^{3}r_{2}/\lambda}$. Then $v(t)$ increases
for $t\le t_{0}$, and we obtain
\begin{gather*}
v(t)\le v(t_{0})<0,\quad\dot{v}(t)\ge\lambda v^{2}(t_{0})/r_{0}^{2}-r_{0}r_{2}>0\quad\forall t\le t_{0}.
\end{gather*}
This yields
\begin{gather*}
\intop_{t}^{t_{0}}\dot{v}(s)\mathrm{d}s\ge\left[\lambda v^{2}(t_{0})/r_{0}^{2}-r_{0}r_{2}\right](t_{0}-t)
\end{gather*}
 and as a consequence
\begin{alignat*}{1}
v(t) & \le v(t_{0})+\left[v^{2}(t_{0})/r_{0}^{2}-r_{0}r_{2}\right](t-t_{0})\to-\infty,\quad t\to-\infty,\\
u(t) & =u(t_{0})-\intop_{t}^{t_{0}}v(s)\mathrm{d}s\ge u(t_{0})-v(t_{0})(t_{0}-t)\to+\infty,\quad t\to-\infty.
\end{alignat*}
We again arrive at contradiction.

Observe that if for some $\varepsilon>0$ there exists a segment $[t_{1},t_{2}]$
where $\left\Vert \dot{x}(t)\right\Vert ^{2}\ge\left(r_{0}r_{2}+\varepsilon\right)/\lambda$,
then $\dot{v}(t)>\varepsilon$, and the inequality $v(t_{2})\ge v(t_{1})+\varepsilon(t_{2}-t_{1})$
yields
\begin{gather*}
t_{2}-t_{1}\le\frac{2\sqrt{r_{0}^{3}r_{2}/\lambda}}{\varepsilon}.
\end{gather*}
 Hence, for any $\varepsilon>0$ and any $T>0$ there exists $t_{\varepsilon}<-T$
such that $\left\Vert \dot{x}(t_{\varepsilon})\right\Vert ^{2}<\left(r_{0}r_{2}+\varepsilon\right)/\lambda$.
Now it remains to estimate $\left\Vert \dot{x}(t)\right\Vert $ on
the segment $[t_{1},t_{2}]$ such that $\left\Vert \dot{x}(t_{i})\right\Vert ^{2}=\left(r_{0}r_{2}+\varepsilon\right)/\lambda$
and $\left\Vert \dot{x}(t)\right\Vert ^{2}>\left(r_{0}r_{2}+\varepsilon\right)/\lambda$
for all $t\in(t_{1},t_{2})$.

On account that
\begin{gather*}
2\left\Vert \dot{x}(t)\right\Vert \left|\frac{\mathrm{d}\left\Vert \dot{x}(t)\right\Vert }{\mathrm{d}t}\right|=\left|\frac{\mathrm{d}}{\mathrm{d}t}\left\Vert \dot{x}(t)\right\Vert ^{2}\right|=2\left|\left\langle \dot{x}(t),\nabla_{\dot{x}}\dot{x}(t)\right\rangle \right|\le2r_{2}\left\Vert \dot{x}(t)\right\Vert
\end{gather*}
on any interval where $\dot{x}(t)\ne0$, we obtain
\begin{gather*}
\left|\frac{\mathrm{d}\left\Vert \dot{x}(t)\right\Vert }{\mathrm{d}t}\right|\le r_{2}\quad\forall t\in(t_{1},t_{2}).
\end{gather*}
Set $z(t):=\left\Vert \dot{x}(t)\right\Vert $, $z_{\varepsilon}:=\sqrt{\left(r_{0}r_{2}+\varepsilon\right)/\lambda}$.
Then
\begin{gather*}
\dot{v}(t)\ge\lambda\left[\left\Vert \dot{x}(t)\right\Vert ^{2}-r_{0}r_{2}/\lambda\right]\ge\lambda\left[z^{2}(t)-z_{\varepsilon}^{2}\right]
\end{gather*}
and thus
\begin{gather}
\left|\left[z^{2}(t)-z_{\varepsilon}^{2}\right]\frac{\mathrm{d}z(t)}{\mathrm{d}t}\right|\le r_{2}\left[z^{2}(t)-z_{\varepsilon}^{2}\right]\le\frac{r_{2}}{\lambda}\dot{v}(t)\quad\forall t\in(t_{1},t_{2}).\label{eq:r_2/ldot_v}
\end{gather}
If we define
\[
I(z):=\frac{z^{3}}{3}-z_{\varepsilon}^{2}z+\frac{2z_{\varepsilon}^{3}}{3},
\]
then~one can rewrite(\ref{eq:r_2/ldot_v}) in the form
\begin{gather*}
-\frac{r_{2}}{\lambda}\dot{v}(t)\le\frac{\mathrm{d}}{\mathrm{d}t}I(z(t))\le\frac{r_{2}}{\lambda}\dot{v}(t)\quad\forall t\in(t_{1},t_{2}).
\end{gather*}
From this it follows that
\begin{gather*}
\begin{split}2z_{\varepsilon}^{3} & \ge2\sqrt{r_{0}^{3}r_{2}^{3}/\lambda^{3}}\ge\frac{r_{2}}{\lambda}\left[v(t_{2})-v(t_{1})\right]\\
 & =\frac{r_{2}}{\lambda}\intop_{t_{1}}^{t_{2}}\dot{v}(s)\mathrm{d}s=\frac{r_{2}}{\lambda}\intop_{t_{1}}^{t}\dot{v}(s)\mathrm{d}s+\frac{r_{2}}{\lambda}\intop_{t}^{t_{2}}\dot{v}(s)\mathrm{d}s\\
 & =\intop_{t_{1}}^{t}\frac{\mathrm{d}}{\mathrm{d}s}I(z(s))\mathrm{d}s-\intop_{t}^{t_{2}}\frac{\mathrm{d}}{\mathrm{d}s}I(z(s))\mathrm{d}s=2I(z(t)).
\end{split}
\end{gather*}
\begin{alignat*}{1}
\end{alignat*}
Hence,
\begin{gather*}
\frac{z^{3}(t)}{3}-z_{\varepsilon}^{2}z(t)+\frac{2z_{\varepsilon}^{3}}{3}\le z_{\varepsilon}^{3}\quad\forall t\in(t_{1},t_{2}).
\end{gather*}
Introducing the new variable $\zeta=z/z_{\varepsilon}$, we obtain
\begin{gather*}
\zeta^{3}(t)-3\zeta(t)-1\le0\quad\forall t\in(t_{1},t_{2}),
\end{gather*}
and finally, by letting $\varepsilon$ tend to zero,
\begin{gather*}
z(t)\le C\sqrt{r_{0}r_{2}/\lambda}\quad\forall t\in(t_{1},t_{2}).
\end{gather*}

\end{proof}

\section{Inequality for curves on the unit sphere }

Consider a curve $\mathbf{x}(\cdot)\in\mathrm{C}^{2}\left(\mathbb{R}\!\mapsto\!\mathbb{R}^{3}\right)$
such that $\mathbf{x}(t)\in\mathbb{S}^{2}$:=$\left\{ \mathbf{x}\in\mathbb{R}^{3}:\left\Vert \mathbf{x}\right\Vert =1\right\} $
for all $t\in\mathbb{R}$. Let
\begin{gather}
\mathbf{e}\in\mathrm{argmin\left\{ \left\Vert \mathbf{x}-\mathbf{x}(\cdot)\right\Vert _{\infty}:\mathbf{x}\in\mathbb{S}^{2}\right\} }.\label{eq:Ch-c}
\end{gather}

Define $U(\mathbf{x}):=\left\Vert \mathbf{x}-\mathbf{e}\right\Vert ^{2}/2$.
Then for $\quad\mathbf{x}\in\mathbb{S}^{2}$ we have
\begin{gather*}
\nabla U(\mathbf{x})=\mathbf{x}-\mathbf{e}-\left\langle \mathbf{x}-\mathbf{e},\mathbf{x}\right\rangle \mathbf{x}=\left\langle \mathbf{e},\mathbf{x}\right\rangle \mathbf{x}-\mathbf{e},\\
\left\Vert \nabla U(\mathbf{x})\right\Vert ^{2}=\left\langle \left\langle \mathbf{e},\mathbf{x}\right\rangle \mathbf{x}-\mathbf{e},\left\langle \mathbf{e},\mathbf{x}\right\rangle \mathbf{x}-\mathbf{e}\right\rangle =1-\left\langle \mathbf{e},\mathbf{x}\right\rangle ^{2}
\end{gather*}
Since a naturally parametrized geodesic on $\mathbb{S}^{2}$ starting
at $\mathbf{x}_{0}\in\mathbb{S}^{2}$ in direction of a vector $\mathbf{y}\in T_{\mathbf{x}_{0}}\mathbb{S}^{2}$
is a solution $\mathbf{x=\mathbf{g}}(t)$ of initial problem$ $ $ $
\begin{gather*}
\ddot{\mathbf{x}}+\left\langle \dot{\mathbf{x}},\dot{\mathbf{x}}\right\rangle \mathbf{x}=0,\quad\mathbf{x}(0)=\mathbf{x}_{0},\quad\dot{\mathbf{x}}(0)=\mathbf{y},
\end{gather*}
then
\begin{alignat*}{1}
\left\langle \nabla_{\mathbf{y}}\nabla U(\mathbf{x}_{0}),\mathbf{y}\right\rangle  & =\frac{\mathrm{d^{2}}}{\mathrm{d}t^{2}}\biggl|_{t=0}U(\mathbf{g}(t))\\
 & =\left[\left\langle \mathbf{x}-\mathbf{e},\ddot{\mathbf{x}}\right\rangle +\left\langle \dot{\mathbf{x}},\dot{\mathbf{x}}\right\rangle \right]_{\mathbf{x}=\mathbf{g}(t)}\bigl|_{t=0}=\left\langle \mathbf{e},\mathbf{x}_{0}\right\rangle .
\end{alignat*}

Applying Theorem~\ref{thm:1} we obtain the following result.
\begin{thm}
Let $\mathbf{x}(\cdot)\in\mathrm{C}^{2}\left(\mathbb{R}\!\mapsto\!\mathbb{S}^{2}\right)$
and let $\mathbf{e}$ is defined by~(\ref{eq:Ch-c}). Suppose that
\begin{gather*}
\lambda:=\inf_{t\in\mathbb{R}}\left\langle \mathbf{x}(t),\mathbf{e}\right\rangle >0.
\end{gather*}
 Then
\begin{gather*}
\left\Vert \dot{\mathbf{x}}(\cdot)\right\Vert _{\infty}\le\frac{C^{2}}{\lambda}\left\Vert \sqrt{1-\left\langle \mathbf{e},\mathbf{x}(\cdot)\right\rangle ^{2}}\right\Vert _{\infty}\left\Vert \nabla_{\dot{\mathbf{x}}}\dot{\mathbf{x}}(\cdot)\right\Vert _{\infty}\le\frac{C^{2}\sqrt{1-\lambda^{2}}}{\lambda}\left\Vert \nabla_{\dot{x}}\dot{\mathbf{x}}(\cdot)\right\Vert _{\infty}\end{gather*}
where $C$ is defined in Theorem~(\ref{thm:1}).
\end{thm}

\section{Concluding remarks}
\begin{enumerate}
\item It seems natural to use for the auxiliary function $U(\cdot)$$ $ the
function $\rho^{2}(x_{\ast},\cdot)/2$ where $x_{\ast}$ is an element
of Chebyshev center for a given curve $x(\cdot)$.
\item It is still unclear whether our approach is applicable to the case
where $\mathbb{R}$ is replaced by a finite interval or by $\mathbb{R}_{+}$.
\item In the present paper, we leave aside the question of how close is
constant $C\approx1.87939$ to the best possible.
\item Let $I=\mathbb{R}$ or $I=\mathbb{R}_{+}$, and let $k,n\in\mathbb{N}$,
$1\le k<n$. It is known (see, e.g.,\cite{MPF91,Finch03}) that there
exist positive constants $C(I,n,k)$ such that if a function $f(\cdot)\in\mathrm{C}^{n}\left(I\!\mapsto\!\mathbb{R}\right)$
satisfies the boundedness conditions
\begin{gather*}
\left\Vert f(\cdot)\right\Vert _{\infty,I}:=\sup_{t\in I}\left|f(t)\right|<\infty,\quad\left\Vert f^{(n)}(\cdot)\right\Vert _{\infty,I}<\infty,
\end{gather*}
then
\begin{gather*}
\left\Vert f^{(k)}(\cdot)\right\Vert _{\infty,I}\le C(I,n,k)\left\Vert f(\cdot)\right\Vert _{\infty,I}^{1-k/n}\left\Vert f^{(n)}(\cdot)\right\Vert _{\infty,I}^{k/n}.
\end{gather*}
In this context, there arises a natural question whether it is possible
to obtain the inequalities of the form
\begin{gather*}
\left\Vert \nabla_{\dot{x}}^{k-1}\dot{x}(\cdot)\right\Vert _{\infty,I}\le K(I,n,k)\left\Vert \rho(x_{\ast},x(\cdot))\right\Vert _{\infty,I}^{1-k/n}\left\Vert \nabla_{\dot{x}}^{n-1}\dot{x}(\cdot)\right\Vert _{\infty,I}^{k/n}
\end{gather*}
for bounded mappings $x(\cdot):I\mapsto\mathcal{M}$ with bounded
$(n-1)$-iterate of covariant derivative $\nabla_{\dot{x}}$ .
\end{enumerate}

\subsection*{Acknowledgements }

This work was partially supported by the Ministry of Education and
Science of Ukraine {[}project 0116U004752{]}.

\end{document}